\documentclass[11pt,a4paper,reqno]{amsart}

\usepackage[left=30mm, right=30mm, bottom=25mm, top=25mm,includefoot]{geometry}

\usepackage{color} 
\usepackage{xcolor}
\usepackage{hyperref}

\usepackage{tikz}
\usepackage{pgf}
\usetikzlibrary{
    shapes.symbols,shapes.geometric, cd,
    shadows,arrows.meta,
    graphs, calc, hobby, decorations.pathreplacing, 
    decorations.markings, matrix, patterns
}
\usepgfmodule{decorations}

\tikzstyle{EDR}=[draw=black,line width=1pt,preaction={clip, postaction={pattern=north west lines, pattern color=black}}]

\usepackage{etoolbox}
\makeatletter
\patchcmd\maketitle
  {\@nx\MakeUppercase{\the\toks@}}
  {\the\toks@}
  {}
  {}{}
\patchcmd\@setauthors
  {\MakeUppercase{\authors}}
  {\authors}
  {}{}
\makeatother

\usepackage{array}
\usepackage{float}

\usepackage[tableposition=below]{caption}

\makeatletter
\newcommand*\PrintSkips[1]{%
  \typeout{In #1:}%
  \typeout{\@spaces above: \the\abovecaptionskip}%
  \typeout{\@spaces below: \the\belowcaptionskip}%
}
\makeatother

\usepackage{amsmath}
\usepackage{amssymb}
\usepackage{extarrows}
\usepackage{braket}
\usepackage{mathtools}
\usepackage{commath}
\usepackage{mathrsfs}

\DeclarePairedDelimiter{\ceil}{\lceil}{\rceil}

\makeatletter
\newsavebox{\@brx}
\newcommand{\llangle}[1][]{\savebox{\@brx}{\(\m@th{#1\langle}\)}%
\mathopen{\copy\@brx\kern-0.5\wd\@brx\usebox{\@brx}}}
\newcommand{\rrangle}[1][]{\savebox{\@brx}{\(\m@th{#1\rangle}\)}%
\mathclose{\copy\@brx\kern-0.5\wd\@brx\usebox{\@brx}}}
\makeatother

\usepackage{amsthm}

\theoremstyle{plain}
\newtheorem{Th}{Theorem}[section]

\newtheorem{Lm}[Th]{Lemma}
\newtheorem*{Lm_orderings}{Lemma~3.2$'$}

\newtheorem{Corollary}[Th]{Corollary}

\theoremstyle{definition}
\newtheorem{Example}[Th]{Example}
\newtheorem{Def}[Th]{Definition}
\newtheorem{Remark}[Th]{Remark}

\makeatletter
\renewenvironment{proof}[1][Proof]{
  \par
  \pushQED{\qed}%
  \normalfont
  \topsep0pt \partopsep3pt
  \trivlist
  \item[\hskip\labelsep
        \itshape
        #1\@addpunct{.}]\ignorespaces
}{
  \popQED\endtrivlist\@endpefalse
  \addvspace{6pt plus 6pt} %
}
\makeatother

\newcommand{\conv}[1]{\mathrm{Conv}\left(#1\right)}
\newcommand{\eps}[0]{\varepsilon}
\newcommand{\al}[0]{\alpha}
\newcommand{\R}[0]{\mathbb{R}}
\DeclareMathOperator{\dist}{dist}

\begin{document}

\title[\resizebox{4in}{!}{Logarithmic algorithms for fair division problems}]{Logarithmic algorithms for fair division problems}

\author[A. Grebennikov]{Alexandr Grebennikov}
\address{A.~Grebennikov: Saint-Petersburg State University, Department of Mathematics and Computer Sciences, 14th Line 29B, Vasilyevsky Island, St.\,Petersburg 199178, Russia}
\email{sagresash@yandex.ru}

\author[X. Isaeva]{Xenia Isaeva}
\address{X.~Isaeva: HSE University, department of mathematics, Usacheva str. 6,  Moscow 119048, Russia}
\email{ksisaeva@edu.hse.ru}

\author[A. V. Malyutin]{Andrei V. Malyutin}
\address{A.\,V.~Malyutin: St.\,Petersburg Department of Steklov Institute of Mathematics, 
Fontanka 27, St.\,Petersburg 191011, Russia}
\email{malyutin@pdmi.ras.ru}

\author[M. Mikhailov]{Mikhail Mikhailov}
\address{M.~Mikhailov: St Petersburg University, department of mathematics and mechanics, University avenue 28, St.\,Petersburg 198504, Russia}
\email{mikhailovmikhaild@yandex}

\author[O. R. Musin]{Oleg R. Musin}
\address{O.\,R.~Musin: University of Texas Rio Grande Valley, One West University Boulevard, Brownsville, TX, 78520} 
\email{oleg.musin@utrgv.edu}

\maketitle




\begin{abstract}
    We study the algorithmic complexity of fair division problems with a focus on minimizing the number of queries needed 
    to find an approximate solution with desired accuracy.
    We show for several classes of fair division problems that under certain natural conditions on sets of preferences, a logarithmic number of queries with respect to accuracy is sufficient.
\end{abstract}


\thanks{Mathematics Subject Classification 91B32 68Q17 52C45}

\section{Introduction}
\label{sec:introduction}

The problem of fair division is old and famous, it has many forms and arises in numerous real-world settings. See, e.\,g.,~ \cite{Robertson98} and references therein for an introduction to the subject. 
Among the many interesting aspects of this problem, the area of algorithmic issues stands out.
The state of art here is well characterized by the following words of~\cite{Robertson98}:
``\emph{The main unresolved issue ... is the general question of finding bounded finite algorithms for envy-free division for any number of players. This is a well-known problem, and any algorithm for envy-free division, even for special small cases ... would be of interest}.''

In this paper we investigate fast algorithms for some particular cases of the rental harmony and cake-cutting problems. 
To recall the former, suppose a group of friends consider renting a house but they shall first agree on how to allocate its rooms and share the rent. They will rent the house only if they can find a room assignment-rent division which appeals to each of them. 
We call this problem the {\it rental harmony} problem, following Su~\cite{Su99}. 
Our investigation of the rental harmony problem starts with the simplest situation when each of the tenants forms for their own opinion  of a fair price for each of the rooms, according to the tenant's own criteria, and is ready to rent a room at this or lower price, regardless (say, not knowing) how the rest of the rent is distributed among other tenants.

In the envy-free {\it cake-cutting problem}, a ``cake'' (a heterogeneous divisible resource) has to be divided among $d$ partners with different preferences over parts of the cake \cite{DS,Stn,Strom,Su99}. 
An alternative interpretation of this problem is as several subcontractors distribute parts of some job among themselves.
In particular, to emphasize the duality with the rental harmony problem, suppose a group of workers discussing renovating a house but they shall first agree on who is in charge of which room and allocate the reward.
The simplest version for this problem is when each of the workers forms their own opinion of a fair price for repairing each of the rooms, according to the worker's own criteria, and is ready to renovate a room for this or greater fee.

A convenient visual-geometric point of view 
on the fair division problems is provided by the configuration space approach.
In case of $d$ agents (friends/ tenants/ partners/ subcontractors/ workers) the configuration space of all possible rent/reward allocations is the standard $(d-1)$-dimensional simplex $A=\Delta^{d-1}$, that is, all representations of a positive number as a sum of $d$ nonnegative ones. 
The preference sets of agents are then subsets of this simplex, the conditions of existence of a fair division can be formulated as conditions on these sets, and solution existence theorems are related to extensions of the KKM (Knaster--Kuratowski--Mazurkiewicz) lemma.\footnote{The existence of a solution to envy-free division problems was discussed in \cite{Asada18, Bapat, Frick19, Gale, Mus17, Su99}. Gale~\cite{Gale} proved a colorful (rainbow) version of the KKM lemma that can be applied for the existence theorem proved by Su~\cite{Su99}.}
The mentioned above simplest versions are the cases where these subsets are intersections of the simplex with half-spaces whose boundaries are parallel to the faces of the simplex.
Another case we consider is of convex preference sets.

In terms of configuration spaces, the fair division problems we study are formulated as follows. 
There is a simplex and collections of preference sets (satisfying conditions of solution existence). 
We do not know these sets, but we can ask which of these sets contain any chosen point of the simplex.
We are interested in algorithms for finding (approximate) solutions aimed at minimizing the number of queries.

We consider two types of queries.
One of these two means that we choose a preference set and a point and check whether this set contains this point.
For example, for the rental harmony problem this means that a tenant says if a given room suits them under a given rent distribution.
We call this type of query the \emph{binary mode}.
Another approach, which we refer to as the \emph{minimal mode}, is when a tenant indicates one of the rooms that suits them (under a given rent distribution).

Observe that a minimal mode query can be simulated by $d-1$ binary mode queries.\footnote{This is because if none of $d-1$ preference sets of a person covers a point in the configuration space, then without running the $d$th query we know that the point is covered by the person's $d$th preference set.} This implies that any algorithm that finds an approximate solution in minimal mode using at most $N$ queries is obviously modified into an algorithm that finds such a solution in binary mode using at most $(d-1)N$ queries. In non-degenerate cases, no minimal mode queries can completely recover the information obtained by binary mode queries (to see this, consider examples where each of the preference sets covers almost the entire configuration space).

Another important issue is that we are not looking for a point close to a solution, but a point that is a solution when all preference sets increase by a prescribed $\eps$ (in a natural metric on the simplex). 
Such a point is called an \emph{$\eps$-fair division point}.
Note that an $\eps$-fair division point may be located more than $\eps$ away from solutions. 
The following figure shows an example of a three-tenant rental harmony problem (with each tenant having the same triplet of convex preference sets), where an $0.1$-fair division point $E$ is more than $0.5$ away from the (unique) solution point~$S$.

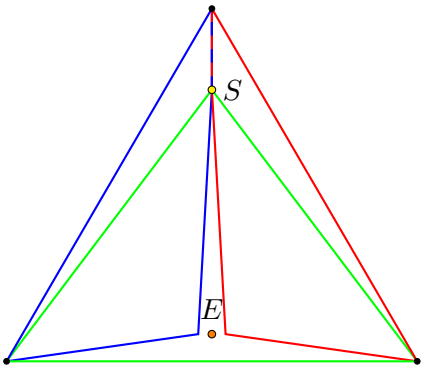
\begin{figure}[h!]
    \centering   
            \begin{minipage}[h]{0.3\linewidth}\centering\begin{tikzpicture}
        [
        scale=1.8,
        >=stealth,
        point/.style = {draw, circle,  fill = black, inner sep = 1pt},
        dot/.style   = {draw, circle,  fill = black, inner sep = .75pt},
        use Hobby shortcut
        ]
        
        \def\xang{0.500}
        \def\yang{0.866}
        \coordinate (ptA) at (0, 0);
        \coordinate (ptB) at (3, 0);
        \coordinate (ptC) at (3*\xang, 3*\yang);
        
        \coordinate (CCC) at (1.5, 2);
        \coordinate (LL) at (1.4, 0.2);
        \coordinate (RR) at (1.6, 0.2);        
        
        \draw[blue, thick]  (ptA) -- (LL) -- (CCC) -- (ptC);
        \draw[red, thick]   (ptB) -- (RR) -- (CCC) -- (ptC);
        \draw[green, thick] (ptA) -- (CCC) -- (ptB); 
        \draw[blue, thick]  (CCC) -- (1.5, 2.1);
        \draw[blue, thick]  (1.5, 2.2) -- (1.5, 2.3);
        \draw[blue, thick]  (1.5, 2.4) -- (1.5, 2.5);
        
        \draw[blue, thick]  (ptA) -- (ptC);
        \draw[red, thick]   (ptB) -- (ptC);
        \draw[green, thick] (ptA) -- (ptB);
        
        \node (A) at (ptA) [dot]{};
        \node (B) at (ptB) [dot]{};
        \node (C) at (ptC) [dot]{};
        
        \node (X) at (CCC) [draw, circle,  fill = yellow, inner sep = 1pt]{};
        \node (X) at (CCC) [label = right:\!\!$S$]{};
        \node (Y) at (1.5, 0.2) [draw, circle,  fill = orange, inner sep = 1pt]{};
        \node (Y) at (1.5, 0.16) [label = above:$E$]{};
        
        \end{tikzpicture}\end{minipage}

    \caption{An example with a unique solution $S$ and a $0.1$-fair division point $E$ located at a large distance from $S$.}
    \label{fig:exampleIntro}
\end{figure}


It is known \cite{Strom08} that in general, there is no finite protocol
for finding a fair division point for envy-free cake-cutting problem for $d\ge 3$ people with $d-1$ cuts, and Theorem~4 in~\cite{Deng12} says that finding an $\eps$-fair division point in this case is PPAD-complete.
Taking the configuration space point of view, we immediately see that in binary mode,
an $\eps$\nobreakdash-fair division point can be found in $O(n^{d-1})$ queries, where $n$ stands for~$\ceil{1/\eps}$. 
Indeed, it is enough to check all the vertices of some triangulation mesh size less than~$\eps$. 
It is shown in~\cite[Theorem~5]{Deng12} that for the cake division problem, $O(n^{d-2})$ queries are sufficient for $d\ge3$ in the minimal mode (and hence in the binary mode).

The main point of the present article is that, under certain natural conditions on preference sets, a logarithmic number of queries with respect to accuracy is sufficient (provided we know in advance that the preference sets satisfy the conditions).
In particular, in the case when all preference sets are intersections of the simplex with half-spaces of a certain kind (motivated by the cake division problem) Theorem~\ref{pie} says that $(d-1)^2\ceil{\log_2 (n \cdot (d-1))}$ queries suffice in binary mode,\footnote{We conjecture that $(d-1)\ceil{\log_{\frac{d}{d-1}} 2n}$ queries suffice in minimal mode for this case (see Remark~\ref{rem:pie-min}).}
and in the case when all preference sets are intersections with half-spaces of another specific kind (motivated by the rent division problem) 
Theorem~\ref{rent-binary-direct} and Corollary~\ref{rent-binary} imply that 
$$(d-1)^2 \min \{\ceil{\log_2 (n \cdot (d-1))}, \ceil{\log_{\frac{d}{d-1}} n}\}$$ queries suffice in binary mode, while Theorem~\ref{rent} states that $(d-1)\ceil{\log_{\frac{d}{d-1}} n}$ queries suffice in minimal mode.

In order to compare these estimates, set $E(d):=\big(\frac{d}{d-1}\big)^{d-1}$ and observe that for $d\ge 2$ we have $E(d)\in[2,e)$, which implies that
$$
(d-1)\log_{\frac{d}{d-1}} n = (d-1)^2\log_{E(d)} n \le (d-1)^2\log_2 n \le (d-1)^2\log_2 (n \cdot (d-1)).
$$
So, in all these cases $O(\log(n))$ queries are enough. A~series of related results (with logarithmic query complexity algorithms for other query types and other settings in cake-cutting problems) are presented in~\cite{BrNi, ES-HS}.

If the preference sets are convex and $d=3$, Theorem~\ref{convex} says that $O(\log^2(n))$ queries in binary mode are enough for the rental harmony problem.\footnote{The estimate of Theorem~\ref{convex} is also correct in the minimal mode, but the proof of this fact uses a different algorithm, which will be presented in a forthcoming paper.}
A~similar result concerning the cake-cutting problem is Theorem~6 in~\cite{Deng12}.


\section{Stating the problem}
\label{Sec: Stating the problem}

Let $d$ be a positive integer\footnote{In those situations where the case $d = 1$ is degenerate, we will assume by default that $d\ge 2$.}, 
let $A$ be a $(d-1)$-dimensional regular simplex with edges of Euclidean length $1$ in $\R^{d-1}$, 
and let $v_1, \ldots, v_d$ be the vertices of~$A$. 
We use the barycentric coordinate system on~$A$.
Namely, each $x\in A$ can be uniquely expressed in the form
$$
x=\al_1v_1+...+\al_dv_d
$$
with all $\al_i\ge0$ and $\al_1 + \ldots + \al_d = 1$.
We will use the notation $[\al_1, \ldots, \al_d]$ for~$x$.
By $\dist$ we denote the metric defined by
$$
\dist([\al_1, \ldots, \al_d], [\al'_1, \ldots, \al'_d])=\max_{i\in\{1,\ldots,d\}} |\al_i-\al'_i |.
$$
For $j\in\{1, \ldots, d\}$, we denote the facet $\conv{\{v_k\}_{k \neq j}}$ of~$A$ by~$F_j$. 

For the rental harmony problem, $A$~corresponds to all representations of total price as a sum of $d$ nonnegative numbers; and $F_j$ is precisely the set of price distributions with zero price for the $j$th room.
Then allocations of rooms with prices are described by pairs $(\sigma,x)$, where $\sigma$ a permutation of $(1,\ldots,d)$ and $x$ is a point in~$A$. 
A pair $(\sigma,x)$ describes the situation where the $i$th person occupies the $\sigma(i)$th room with a rent of $\alpha_{\sigma(i)}$ of total price, where $\alpha_{\sigma(i)}=\alpha_{\sigma(i)}(x)$ is the $\sigma(i)$th coordinate of~$x$.

Assume that for each $i\in\{1, \ldots, d\}$ a covering $\{A_{i1}, \ldots, A_{id}\}$ of~$A$ consisting of $d$ closed sets is fixed (these coverings are preferences of our agents):
$$
     \bigcup_{j\in\{1,\ldots,d\}} A_{ij} = A,
$$
but the elements of these coverings are \textbf{not known} to us. 

We want to find a way to allocate the rooms and split the total price into $d$ prices of rooms 
such that every person is ``satisfied'' with it in the sense that for the corresponding pair $(\sigma,x)$ we have $x\in A_{i\sigma(i)}$.
We will denote by $\eps$ the maximum ``error'' we allow.
For $\eps \ge 0$, 
we denote the $\eps$-neighbourhoods of the preference sets as 
$$
    A^{\eps}_{ij} = \{x \in \R^{d-1}: \dist(x, A_{ij}) \le \eps \}
$$ 
and the intersections of these $\eps$-neighbourhoods as
$$
    I^{\eps}_i = A^{\eps}_{i1} \cap \ldots \cap A^{\eps}_{id}.
$$

We introduce the following definition.
\begin{Def}
    We say that $x$ in $A$ is an
    \emph{$\eps$-fair division point}
    if there exists a permutation~$\sigma$ of $(1,\ldots,d)$ such that $x \in A^{\eps}_{j\sigma(j)}$ for all $j\in\{1,\ldots,d\}$. 
    We say that $x$ a \emph{fair division point} if $x$ is an $\eps$-fair division point for $\eps = 0$.
\end{Def}

In terms of the rental harmony problem, a price distribution 
is $\eps$-fair (an $\eps$-fair division point) if there is an allocation of rooms
where 
for each person there is an adjustment in room prices (with the same total cost) 
such that
the person is ``satisfied'' with adjusted prices while
the price of each room is adjusted by no more than $\eps$ of the total price. 
(Perhaps no one is satisfied with the initial $\eps$-fair price distribution and different persons need different adjustments.)
An $\eps$-fair division point may be located ``far away'' (more than $\eps$ away) from all fair division points (see Fig.~\ref{fig:exampleIntro}).

General theory provides some reasonable sufficient conditions when a fair division point exists. 
For example, the rainbow KKM lemma tells it exists whenever
$$
    \conv{\{v_j\}_{j \in J}} \subset \bigcup_{j \in J} A_{ij} \text{\; for all\;} i\in \{1, \ldots, d\}, J \subset \{1, \ldots, d\},
$$
and a rainbow generalization of Sperner's lemma implies that it exists whenever
$$
    \conv{\{v_j\}_{j \in J}} \cap \bigcup_{j \notin J}A_{ij}  = \emptyset \text{\; for all\;} i\in \{1, \ldots, d\}, J \subset \{1, \ldots, d\}.
$$

Our goal is to find an $\eps$-fair division point using the smallest possible number of queries. 

Now we give a formal definition for the two types of queries.
One of these two is as follows: we choose an index $i$ and a point $x \in A$ and receive an index $j$ such that $x \in A_{ij}$ as an answer; we call this type of query the
\emph{minimal mode}. 
Another approach is: we choose indices $i$ and $j$ and a point $x \in A$ and receive ``yes'' if $x \in A_{ij}$ and ``no'' otherwise; 
this type of query is called the 
\emph{binary mode}.


\bigskip

\section{Preference sets with the inclusion property}


In this section we prove two auxiliary constructive statements showing that if the collection of preference sets has a specific property (given in  Definition~\ref{def:inclusion property}),
then one can find a solution point (i)~in a set of points each of which is found separately for each of the agents according to their preferences, and (ii)~without information about (the preferences of) one of the agents.\footnote{Of course, to move from a solution point to the corresponding solution itself (for example, to a specific resource allocation between the agents), information about the excluded agent's preferences may also be needed.} 
\footnote{This is directly related to the results of~\cite{Asada18}, see also \cite{Frick19}, where it is proved that there exists an envy-free rent division that remains so for any change in the preferences of one of the tenants.}
Lemma \ref{d_orderings} is a purely combinatorial statement, and Theorem~\ref{d-1_enough} applies this combinatorics to our problem.

The formulation of Lemma~\ref{d_orderings} use many quantifiers and indices, so on the advice of the referee, we preface it with the following warm-up lemma, corresponding to the case ``without a secretive agent''.
(Lemma~\ref{d_orderings_warm_up} is not used in further proofs and is given only for clarity of presentation.)
\begin{Lm}
\label{d_orderings_warm_up}
    Let $\{ \le_1,\ldots, \le_d\}$ 
    be $d$ linear orderings of $D^+ = \{1, \ldots, d\}$. 
    Then there is an element $i_0$ in $D^+$, and a bijective map $\pi\colon D^+ \to D^+$, such that for all $i\in D^+$ we have 
    $$
        i_0 \le_{\pi(i)} i.
    $$
\end{Lm}
\begin{proof}
    We prove the lemma by induction on~$d$. 
    If $d = 1$, the statement is true for $i_0=1$.
    If $d\ge 2$, let $k\in D^+$ be the maximal element with respect to~$\le_1$.
    We apply induction hypothesis to the set $D^+ \setminus \{k\}$ and all of the orderings except $\le_1$, and get some $i_0 \in D^+ \setminus \{k\}$ and a bijective map $$\pi'\colon D^+ \setminus \{k\} \to \{2, \ldots, d\} $$ with the desired property. 
        Since $k$ is maximal with respect to $\le_1$, we can define $\pi$ as
    $$
        \pi(i) = 
        \begin{cases}
            \pi'(i) & \text{if~~}  i \neq k, \\
            1 & \text{if~~}  i = k. 
        \end{cases}
    $$
\end{proof}

\begin{Lm}
\label{d_orderings}
    Let $d\ge 2$, and let $\{ \le_1,\ldots, \le_d \}$ 
    be $d$ linear orderings of $D = \{1, \ldots, d - 1\}$. 
    Then $D$ contains an element $i_0$ such that for each $j_0\in \{1, \ldots, d\}$ there exists a bijective map $\pi\colon D \to \{1, \ldots d\} \setminus \{j_0\}$ with the property that for all $i \in D$ we have
    $$
        i_0 \le_{\pi(i)} i.
    $$
\end{Lm}

Before moving on to the proof, for easier understanding, we give an equivalent indices-free reformulation of the lemma.

\begin{Lm_orderings}
\label{d_orderings_noind}
    Let $m\ge 1$ be an integer, let $E$ be a set of $m$ elements, and let $\mathscr{O}$ be a set of $m+1$ linear orderings of~$E$.
    Then $E$ contains an element $e_*$ such that for each $o \in \mathscr{O}$ there exists a bijective map $\pi_o\colon E \to \mathscr{O} \setminus \{o\}$ with the property that for all $e \in E$ we have
    $$
        e_* \le_{\pi_o(e)} e,
    $$
    where $e_* \le_{\pi_o(e)} e$ means that $e_*$ is not greater than $e$ w.r.t.~$\pi_o(e)$.
\end{Lm_orderings}

\begin{proof}
   We will prove the lemma by induction on~$d$. 
    If $d = 2$, the statement is trivially true for $i_0=1$, so we may assume $d\ge 3$.\footnote{The case $d = 3$, with two elements and three orderings, is also easy to do without induction. 
    In this case, there is an element $k \in D=\{1,2\}$ that is minimal with respect to two distinct orderings, and we obviously obtain what is required by setting $i_0=k$.}
    Since there are $d$ orderings and $d-1$ elements, it follows that there exists an element $k \in D$ that is maximal with respect to two distinct orderings $l_1$ and $l_2$. 
    Now we apply induction hypothesis to the set $D \setminus \{k\}$ and all of the orderings except $l_1$, and get some $i_0 \in D \setminus \{k\} \subset D$. 
    We claim that $i_0$ has the desired property for the initial problem as well.
    
    Suppose we choose (in $\{1, \ldots, d\}$) some $j_0 \neq l_1$. 
    By the choice of $i_0$ there exists a map 
    $$\pi'\colon D \setminus \{k\} \to \{1, \ldots, d\} \setminus \{l_1, j_0\}$$ with the desired property. 
    Since $k$ is maximal with respect to $l_1$, we can define $\pi$ as
    $$
        \pi(i) = 
        \begin{cases}
            \pi'(i) & \text{if~~}  i \neq k, \\
            l_1 & \text{if~~}  i = k. 
        \end{cases}
    $$
    
    In the remaining case with $j_0 = l_1$, by the choice of $i_0$ there exists a map 
    $$\pi'\colon D \setminus \{k\} \to \{1, \ldots, d\} \setminus \{l_1, l_2\}$$ 
    with the desired property. Since $k$ is maximal with respect to $l_2$, it follows that we can define $\pi$ as
    $$
        \pi(i) = 
        \begin{cases}
            \pi'(i) & \text{if~~}  i \neq k, \\
            l_2 & \text{if~~}  i = k. 
        \end{cases}
    $$
\end{proof}

\begin{Def}
\label{def:inclusion property}
    We say that our collection of preference sets 
    $A_{ij}$ satisfies the \emph{inclusion property} if for each triplet $(i_1, i_2, j)$ with $i_1, i_2\in \{1, \ldots, d - 1\}$ and $j\in \{1, \ldots, d\}$ we have either $A_{i_1 j} \subset A_{i_2 j}$ or $A_{i_2 j} \subset A_{i_1 j}$.
\end{Def}

\begin{Th}
\label{d-1_enough}
Suppose the collection of sets $A_{ij}$ satisfies the inclusion property. 
Then each sequence of points $x_1, \ldots, x_{d-1}$ with $x_i \in I_i = A_{i1} \cap \ldots \cap A_{id}$ for each $i\in\{1, \ldots, d-1\}$ contains a fair division point.
Moreover, there exists an algorithm that, given such a sequence and the related data about inclusions, finds a fair division point in the sequence.
\end{Th}
\begin{proof}
    Inclusion orderings on $A_{ij}$ determine orderings on $\{1, \ldots, d-1\}$ by the rule
    $$
         i_1 \le_{j} i_2 \; \Leftrightarrow \; A_{i_1j} \subset A_{i_2j}.
    $$
    We apply Lemma~\ref{d_orderings} to these orderings and obtain some index $i_0\in\{1, \ldots, d - 1\}$ having the property described in Lemma~\ref{d_orderings}. We claim that $x_{i_0}$ is a fair division point.
    
    Indeed, $x_{i_0}$ is in $A_{dj_0}$ for some $j_0$ since $\{A_{d1},\ldots, A_{dd}\}$ is a covering of~$A$.
    Then by the choice of $i_0$ there exists a map 
    $$\pi\colon \{1, \ldots, d - 1\} \to \{1, \ldots, d\} \setminus \{j_0\}$$ such that $i_0 \le_{\pi(i)} i$. Now we define the desired permutation $\sigma$ as
    $$
        \sigma(i) = 
        \begin{cases}
            \pi(i) & \text{if~~}  i\in\{1, \ldots, d - 1\}, \\     
            j_0 & \text{if~~} i = d. \\
        \end{cases}
    $$
    The condition $i_0 \le_{\pi(i)} i$ means that $A_{i_0 \pi(i)} \subset A_{i \pi(i)}$. Then for each $i\in\{1, \ldots, d - 1\}$ we have    
    $$
        x_{i_0} \in A_{i_0 \pi(i)} \subset A_{i \pi(i)} = A_{i \sigma(i)},  
    $$
    and $x_{i_0} \in A_{dj_0} = A_{d \sigma(d)}$ straight from the definition of $j_0$.
\end{proof}


\bigskip

\section{Logarithmic algorithms for the ``linear'' preference sets}
\label{Sec: Half-spaces}

Let $\{A_{ij}\}, \; i=1,...,d, \, j=1,...,d$, as above, be a collection of preference sets. 
In this section we consider two kinds of {\em linear preference sets} (\/{\em LPS}\/):
\begin{itemize}
    \item[(i)] each of $A_{ij}$ has the form $\{[\al_1,\ldots,\al_d] \in A: \al_j \ge a_{ij} \}$; 
    \item[(ii)] each of $A_{ij}$ has the form $\{[\al_1,\ldots,\al_d] \in A: \al_j \le a_{ij} \}$,
\end{itemize}
where $\{a_{ij}\}$ is a set of real numbers.

Case (i) is a particular case of the cake-cutting problem (the number of cuts is~$d-1$; each agent~$i$ accepts the $j$th piece if its length is at least~$a_{ij}$). 
The condition that $\{A_{i1},\ldots,A_{id}\}$ is a covering of~$A$ implies that $\sum_{j=1}^{d} a_{ij} \le 1$ for all~$i$. 
In addition, the standard cake-cutting preferences condition that ``players prefer any piece with mass to an empty piece'' means here that all $a_{ij}$ are in $(0,1)$. 
In (ii) we have a particular case of the rental harmony problem with $\sum_{j=1}^{d} a_{ij} \ge 1$ for all~$i$. Note that in both cases $A_{ij}$ is the intersection of $A$ and a half-space that is bounded by the hyperplane $\al_j=a_{ij}$.

\bigskip

The following theorem is our ``logarithmic complexity'' result for the cake-cutting problem.

\begin{Th}
\label{pie}
    Suppose that each of the sets $A_{ij}$, being a proper subset of the simplex~$A$, is the intersection of~$A$ and a closed half-space with boundary parallel to $F_j$ such that $v_j \in A_{ij}$. Then we can find an $\eps$-fair division point in binary mode using at most $$(d-1)^2\ceil{\log_2 (n \cdot (d-1))}$$ queries, where $n = \ceil{1/\eps}$.
\end{Th}
\begin{proof} 
By assumptions we have LPS case (i). 
    Fix an index $i\in\{1,\ldots,d-1\}$ and denote $\delta = \frac{\eps}{d-1}$. 
    Since $A_{ij}$ are proper subsets of~$A$ and cover~$A$, it follows that $0<a_{ij}<1$.
    For every $j\in\{1,\ldots,d-1\}$ 
    we can use binary search to find approximations $c_{ij}$ such that 
    $$\max(a_{ij} - \delta, 0) \le c_{ij} < a_{ij}$$ 
    with $\ceil{\log_2 \frac{1}{\delta}}$ queries. 
    We set $$x_i = [c_{i1}, c_{i2}, \ldots, c_{i(d-1)}, 1 - c_{i1} - \ldots - c_{i(d-1)}].$$
    The point $x_i$ lies in $A$ since
    $$
    0 \le \sum_{j=1}^{d-1} c_{ij} < \sum_{j=1}^{d-1} a_{ij} \le 1.
    $$

    We spent $(d-1)\ceil{\log_2(n \cdot (d-1))}$ queries on the fixed index~$i$. Repeating this for all 
    $i\in\{1,\ldots,d-1\}$ we obtain $d-1$ points $x_1$, \ldots, $x_{d-1}$ via $(d-1)^2\ceil{\log_2(n \cdot (d-1))}$ queries.

    For $i \in \{1,\ldots,d-1\}$ and $j \in \{1,\ldots,d\}$ we set $c_{id}=1 - c_{i1} - \ldots - c_{i(d-1)}$ and
    $$
        A'_{ij} = \{ [\al_1, \ldots, \al_d]\in A : \al_j \ge c_{ij} \}.
    $$
    Then $x_i \in A'_{i1} \cap \ldots \cap A'_{id}$. The sets $A'_{ij}$ satisfy the inclusion property and are known to us, so by Theorem \ref{d-1_enough} we can find a fair division point $z\in\{x_1, \ldots, x_{d-1}\}$ for them. 
    But we have $A'_{id} \subset  A^{\eps}_{id}$ and $A'_{ij} \subset A^{\delta}_{ij}$ if $j\in\{1,\ldots,d-1\}$, so $z$ is also an $\eps$-fair division point for the initial problem.
\end{proof}
\begin{Example}
    
    The following figure illustrates the first two steps of binary search for finding (an approximation for) \(a_{12}\). The blue dashed line shows the boundary of the actual set \(A_{12}\). The point with the question mark is the point about which we ask whether it lies in \(A_{12}\), and then cut off a piece of \(A\) according to the answer.
    
    
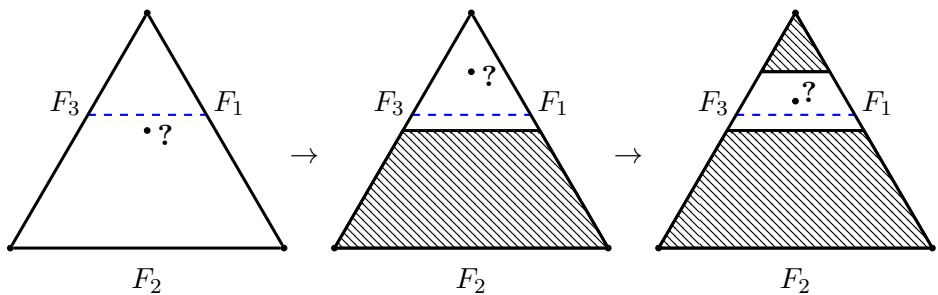
\begin{figure}[H]
        \centering
        \begin{minipage}[h]{0.25\linewidth}\centering\begin{tikzpicture}
        [
        scale=1.2,
        >=stealth,
        point/.style = {draw, circle,  fill = black, inner sep = 1pt},
        dot/.style   = {draw, circle,  fill = black, inner sep = .75pt},
        use Hobby shortcut
        ]
        
        \def\xang{0.500}
        \def\yang{0.866}
        \coordinate (ptA) at (0, 0);
        \coordinate (ptB) at (3, 0);
        \coordinate (ptC) at (3*\xang, 3*\yang);
       
        \draw[blue, dashed, thick] (1.7*\xang, 1.7*\yang) -- (3 - 1.7*\xang, 1.7*\yang);
        
        \node (A) at (0,0)                                            [dot]{};
        \node (B) at (+1.000 + 1.000 + 1.000, +0.000 + 0.000 + 0.000) [dot]{};
        \node (C) at (+\xang + \xang + \xang, +\yang + \yang + \yang) [dot]{}; 
        \node (Q) at (1.5, 1.5 * \yang) [dot]{};
        \node (mark) at (1.7, 1.5*\yang - 0.05) []{\textbf{?}};
        \foreach \from/\to in {A/B, B/C, C/A}
        {
            \draw[line width=0.4mm] (\from) -- (\to);
        }
        \node (F1) at (2, 1.6) [label = right:$F_{1}$]{};
        \node (F2) at (1.5, 0) [label = below:$F_{2}$]{};
        \node (F3) at (1, 1.6) [label = left:$F_{3}$]{};
        \end{tikzpicture}\end{minipage}\(\to\)
        \begin{minipage}[h]{0.25\linewidth}\centering\begin{tikzpicture}
        [
        scale=1.2,
        >=stealth,
        point/.style = {draw, circle,  fill = black, inner sep = 1pt},
        dot/.style   = {draw, circle,  fill = black, inner sep = .75pt},
        use Hobby shortcut
        ]
        
        \def\xang{0.500}
        \def\yang{0.866}
        
        \coordinate (ptA) at (0, 0);
        \coordinate (ptB) at (3, 0);
        \coordinate (ptC) at (3*\xang, 3*\yang);
        
        \draw[blue, dashed, thick] (1.7*\xang, 1.7*\yang) -- (3 - 1.7*\xang, 1.7*\yang);
        \draw[line width=0.4mm]  (1.5*\xang, 1.5*\yang) -- (3 - 1.5*\xang, 1.5*\yang);
        \draw[EDR] (0,0) -- (3, 0) -- (3 - 1.5*\xang, 1.5*\yang) -- (1.5*\xang, 1.5*\yang) -- cycle;
        
        \node (A) at (ptA) [dot]{};
        \node (B) at (ptB) [dot]{};
        \node (C) at (ptC) [dot]{};
        \node (Q) at (1.5, 2.25*\yang) [dot]{};
        \node (mark) at (1.7, 2.25*\yang - 0.05) []{\textbf{?}};
        \foreach \from/\to in {A/B, B/C, C/A}
        {
            \draw[line width=0.4mm] (\from) -- (\to);
        }
        \node (F1) at (2, 1.6) [label = right:$F_{1}$]{};
        \node (F2) at (1.5, 0) [label = below:$F_{2}$]{};
        \node (F3) at (1, 1.6) [label = left:$F_{3}$]{};
        \end{tikzpicture}\end{minipage}\(\to\)
        \begin{minipage}[h]{0.25\linewidth}\centering\begin{tikzpicture}
        [
        scale=1.2,
        >=stealth,
        point/.style = {draw, circle,  fill = black, inner sep = 1pt},
        dot/.style   = {draw, circle,  fill = black, inner sep = .75pt},
        use Hobby shortcut
        ]
        
        \def\xang{0.500}
        \def\yang{0.866}
        \coordinate (ptA) at (0, 0);
        \coordinate (ptB) at (3, 0);
        \coordinate (ptC) at (3*\xang, 3*\yang);

        \draw[blue, dashed, thick] (1.7*\xang, 1.7*\yang) -- (3 - 1.7*\xang, 1.7*\yang);
        
        \draw[line width=0.4mm]  (1.5*\xang, 1.5*\yang) -- (3 - 1.5*\xang, 1.5*\yang);
        \draw[EDR] (0,0) -- (3, 0) -- (3 - 1.5*\xang, 1.5*\yang) -- (1.5*\xang, 1.5*\yang) -- cycle;
        
        \draw[line width=0.4mm]  (2.25*\xang, 2.25*\yang) -- (3 - 2.25*\xang, 2.25*\yang);
        \draw[EDR] (3*\xang, 3*\yang) -- (3 - 2.25*\xang, 2.25*\yang) -- (2.25*\xang, 2.25*\yang) -- cycle;
        
        \node (A) at (ptA) [dot]{};
        \node (B) at (ptB) [dot]{};
        \node (C) at (ptC) [dot]{};
        
        \node (Q) at (1.5, 1.875*\yang) [dot]{};
        \node (mark) at (1.65, 1.875*\yang + 0.1) []{\textbf{?}};
        \foreach \from/\to in {A/B, B/C, C/A}
        {
            \draw[line width=0.4mm] (\from) -- (\to);
        }
        \node (F1) at (2, 1.6) [label = right:$F_{1}$]{};
        \node (F2) at (1.5, 0) [label = below:$F_{2}$]{};
        \node (F3) at (1, 1.6) [label = left:$F_{3}$]{};
        \end{tikzpicture}\end{minipage}
    \end{figure}
\end{Example}

The following Theorem~\ref{rent-binary-direct} is dual to Theorem~\ref{pie} (cake-cutting problem / rental harmony problem), its proof uses arguments from the proof of Theorem~\ref{pie}.
\begin{Th}
\label{rent-binary-direct}
    Suppose that each of the sets $A_{ij}$, being a proper subset of the simplex~$A$, is the intersection of~$A$ and a closed half-space with boundary parallel to $F_j$ such that $F_j \subset A_{ij}$. Then we can find an $\eps$-fair division point in binary mode using at most $$(d-1)^2\ceil{\log_2 (n \cdot (d-1))}$$ queries, where $n = \ceil{1/\eps}$.
\end{Th}
\begin{proof} 
By assumptions we have LPS case (ii). 
Since $A_{ij}$ are proper subsets of~$A$, it follows that $0\le a_{ij}<1$.
Similar to the proof of Theorem~\ref{pie}, we set $\delta = \frac{\eps}{d-1}$ and use binary search to find approximations $c_{ij}$ such that 
    $$a_{ij} < c_{ij} \le \min(a_{ij} + \delta, 1)$$ for all 
    $i\in\{1,\ldots,d-1\}$ and $j\in\{1,\ldots,d-1\}$
    with $(d-1)^2\ceil{\log_2 \frac{1}{\delta}}$ queries. 
   
    In LPS case~(ii), in contrast to LPS case~(i), the inequality $\sum_{1 \le j \le d - 1} c_{ij} \le 1$ may not hold, and we consider two subcases:
    \begin{itemize}
    \item[(a)] $\sum_{1 \le j \le d - 1} c_{ij} \le 1$; 
    \item[(b)] $\sum_{1 \le j \le d - 1} c_{ij} > 1$.
\end{itemize}

    In subcase~(a), the proof repeats verbatim that of Theorem~\ref{pie}, up to replacing $\al_j \ge c_{ij}$ with $\al_j \le c_{ij}$ in the definition of $A'_{ij}$.
    
    In subcase~(b), we set $c_{id}=0$, $S_i=\sum_{1 \le j \le d - 1} c_{ij}$,
    $$
    x_i = [c_{i1}/S_i, c_{i2}/S_i, \ldots, c_{i(d-1)}/S_i, 0],
    $$
    and
    $$
        A'_{ij} = \{ [\al_1, \ldots, \al_d]\in A : \al_j \le c_{ij}/S_i \}.
    $$
    We obviously have $x_i \in A'_{i1} \cap \ldots \cap A'_{id}$. 
    As in the proof of Theorem~\ref{pie}, the sets $A'_{ij}$ satisfy the inclusion property and are known to us, so by Theorem \ref{d-1_enough} we can find a fair division point $z\in\{x_1, \ldots, x_{d-1}\}$ for them. 
    But we have $A'_{id} = F_d\subset  A_{id}$ and $A'_{ij} \subset A^{\delta}_{ij}$ because $c_{ij} / S_i < c_{ij} \le a_{ij} + \delta$, so $z$ is also a $\delta$-fair (hence $\eps$-fair) division point for the initial problem.      
\end{proof}

\begin{Th}
\label{rent}
    Suppose that each of the sets $A_{ij}$ is the intersection of the simplex~$A$ and a closed half-space with boundary parallel to~$F_j$ such that $F_j\subset A_{ij}$. Then we can find an $\eps$-fair division point in minimal mode using at most $(d-1)\ceil{\log_{\frac{d}{d-1}} n}$ queries, where $n = \ceil{1/\eps}$.
\end{Th}
\begin{proof} We have LPS case (ii). 
    We fix an index $i_0\in\{1,\ldots,d-1\}$ and ask which of $A_{i_01}$, \dots, $A_{i_0d}$ contains the center~$c_0$ of~$A$. 
    Some $A_{i_0j_0}$ contains~$c_0$, so we draw the hyperplane~$H_{j_0}$ parallel to $F_{j_0}$ through~$c_0$, and $H_{j_0}\cap A$ also lies in $A_{i_0j_0}$. 
    Clearly, this reduces the problem to the smaller regular simplex~$A'$ that $H_{j_0}$ cuts off from~$A$.
    (Indeed, if we denote by $A'_{ij}$ the intersection~$A'\cap A_{ij}$ and
    $F'_j$ is the facet of~$A'$ that is parallel to~$F_j$, then $A'_{ij}$ is the intersection of~$A'$ and a closed half-space with boundary parallel to~$F'_j$ such that $F'_j\subset A'_{ij}$. This condition copies the assumption of the theorem.)
    This new simplex is homothetic to the initial one with coefficient $\frac{d-1}{d}$. 
    So, after $\ceil{\log_{\frac{d}{d-1}} n}$ queries the resulting regular simplex~$A^\dagger$ has edges of length at most~$\eps$. 
    Then $I^{\eps}_{i_0}$ contains $A^\dagger$, and we take any point of $A^\dagger$ as $x_{i_0}$.
    
    Repeating the described procedure for all indices $i\in\{1,\ldots,d-1\}$ and each time using exactly the specified number of queries, 
    we obtain $d-1$ points lying in the corresponding intersections $I^{\eps}_i$. 
    
    Analogously to the proof of \ref{pie}, for $i \in \{1,\ldots,d-1\}$ and $j \in \{1,\ldots,d\}$ we set
    $$
        A'_{ij} = \{ [\al_1, \ldots, \al_d]\in A : \al_j \le x_{ij} \}, 
        \quad\text{where $x_{ij}$ is the $j$th barycentric coordinate of $x_i$}.
    $$
    These sets satisfy the inclusion property and are known to us, so by Theorem~\ref{d-1_enough} we can find a fair division point $z$ for them. But $x_i \in I^{\eps}_i$ implies that $A'_{ij} \subset A^{\eps}_{ij}$, so $z$ is also an $\eps$-fair division point for the initial problem.
\end{proof}
\begin{Example}
    The following figure illustrates the first two steps of binary search for finding \(x_1 \in I^{\eps}_1\). The blue dashed lines show the boundaries of actual sets $A_{11}$, $A_{12}$, and~$A_{13}$. 
    The point with the question mark is the point about which we ask which of $A_{11}$, $A_{12}$, and~$A_{13}$ contains it.
    At the first step we find out that the point lies in \(A_{13}\), and then cut off a piece of \(A\) according to the answer. Similarly, at the second step we get \(A_{12}\).
    
        
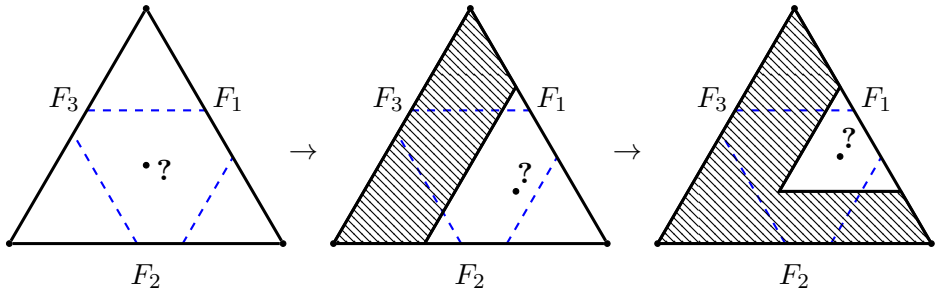
\begin{figure}[H]
        \centering
        \begin{minipage}[h]{0.25\linewidth}\centering\begin{tikzpicture}
        [
        scale=1.2,
        >=stealth,
        point/.style = {draw, circle,  fill = black, inner sep = 1pt},
        dot/.style   = {draw, circle,  fill = black, inner sep = .75pt},
        use Hobby shortcut
        ]
        
        \def\xang{0.500}
        \def\yang{0.866}
        \coordinate (ptA) at (0, 0);
        \coordinate (ptB) at (3, 0);
        \coordinate (ptC) at (3*\xang, 3*\yang);
       
        \draw[blue, dashed, thick] (1.7*\xang, 1.7*\yang) -- (3 - 1.7*\xang, 1.7*\yang);
        \draw[blue, dashed, thick] (1.4, 0) -- (1.4*\xang, 1.4*\yang);
        \draw[blue, dashed, thick] (3 - 1.1, 0) -- (3 - 1.1*\xang, 1.1*\yang);
        
        \node (A) at (0,0)                                            [dot]{};
        \node (B) at (+1.000 + 1.000 + 1.000, +0.000 + 0.000 + 0.000) [dot]{};
        \node (C) at (+\xang + \xang + \xang, +\yang + \yang + \yang) [dot]{}; 
        \node (Q) at (1.5, \yang) [dot]{};
        \node (mark) at (1.7, \yang - 0.05) []{\textbf{?}};
        \foreach \from/\to in {A/B, B/C, C/A}
        {
            \draw[line width=0.4mm] (\from) -- (\to);
        }
        \node (F1) at (2, 1.6) [label = right:$F_{1}$]{};
        \node (F2) at (1.5, 0) [label = below:$F_{2}$]{};
        \node (F3) at (1, 1.6) [label = left:$F_{3}$]{};
        \end{tikzpicture}\end{minipage}\(\to\)
        \begin{minipage}[h]{0.25\linewidth}\centering\begin{tikzpicture}
        [
        scale=1.2,
        >=stealth,
        point/.style = {draw, circle,  fill = black, inner sep = 1pt},
        dot/.style   = {draw, circle,  fill = black, inner sep = .75pt},
        use Hobby shortcut
        ]
        
        \def\xang{0.500}
        \def\yang{0.866}
        
        \coordinate (ptA) at (0, 0);
        \coordinate (ptB) at (3, 0);
        \coordinate (ptC) at (3*\xang, 3*\yang);
        
        \draw[blue, dashed, thick] (1.7*\xang, 1.7*\yang) -- (3 - 1.7*\xang, 1.7*\yang);
        \draw[blue, dashed, thick] (1.4, 0) -- (1.4*\xang, 1.4*\yang);
        \draw[blue, dashed, thick] (3 - 1.1, 0) -- (3 - 1.1*\xang, 1.1*\yang);
        \draw[line width=0.4mm]  (1, 0) -- (2, 2*\yang);
        \draw[EDR] (0,0) -- (1, 0) -- (2, 2*\yang) -- (3*\xang, 3*\yang) -- cycle;
        
        \node (A) at (ptA) [dot]{};
        \node (B) at (ptB) [dot]{};
        \node (C) at (ptC) [dot]{};
        \node (Q) at (2, 2*\yang / 3) [dot]{};
        \node (mark) at (2.1, 2*\yang / 3 + 0.2) []{\textbf{?}};
        \foreach \from/\to in {A/B, B/C, C/A}
        {
            \draw[line width=0.4mm] (\from) -- (\to);
        }
        \node (F1) at (2, 1.6) [label = right:$F_{1}$]{};
        \node (F2) at (1.5, 0) [label = below:$F_{2}$]{};
        \node (F3) at (1, 1.6) [label = left:$F_{3}$]{};
        \end{tikzpicture}\end{minipage}\(\to\)
        \begin{minipage}[h]{0.25\linewidth}\centering\begin{tikzpicture}
        [
        scale=1.2,
        >=stealth,
        point/.style = {draw, circle,  fill = black, inner sep = 1pt},
        dot/.style   = {draw, circle,  fill = black, inner sep = .75pt},
        use Hobby shortcut
        ]
        
        \def\xang{0.500}
        \def\yang{0.866}
        \coordinate (ptA) at (0, 0);
        \coordinate (ptB) at (3, 0);
        \coordinate (ptC) at (3*\xang, 3*\yang);

        \draw[blue, dashed, thick] (1.7*\xang, 1.7*\yang) -- (3 - 1.7*\xang, 1.7*\yang);
        \draw[blue, dashed, thick] (1.4, 0) -- (1.4*\xang, 1.4*\yang);
        \draw[blue, dashed, thick] (3 - 1.1, 0) -- (3 - 1.1*\xang, 1.1*\yang);
        
        
        \draw[line width=0.4mm]  (3 - 2*\xang/3, 2*\yang/3) -- (1 + 2*\xang/3, 2*\yang / 3) -- (2, 2*\yang);
        \draw[EDR]  (3 - 2*\xang/3, 2*\yang/3) -- (1 + 2*\xang/3, 2*\yang / 3) -- (2, 2*\yang) -- (3*\xang, 3*\yang) -- (0, 0) -- (3, 0) -- cycle;
        \node (A) at (ptA) [dot]{};
        \node (B) at (ptB) [dot]{};
        \node (C) at (ptC) [dot]{};
        
        \node (Q) at (2, 2*\yang / 3 + 4*\yang / 9) [dot]{};
        \node (mark) at (2.1, 2*\yang / 3 + 4*\yang / 9 + 0.2) []{\textbf{?}};
        \foreach \from/\to in {A/B, B/C, C/A}
        {
            \draw[line width=0.4mm] (\from) -- (\to);
        }
        \node (F1) at (2, 1.6) [label = right:$F_{1}$]{};
        \node (F2) at (1.5, 0) [label = below:$F_{2}$]{};
        \node (F3) at (1, 1.6) [label = left:$F_{3}$]{};
        \end{tikzpicture}\end{minipage}
    \end{figure}
\end{Example}

\begin{Corollary}
\label{rent-binary}
    Suppose that each of the sets $A_{ij}$ is the intersection of the simplex~$A$ and a closed half-space with boundary parallel to~$F_j$ such that $F_j\subset A_{ij}$. Then we can find an $\eps$-fair division point in binary mode using at most $(d-1)^2\ceil{\log_{\frac{d}{d-1}} n}$ queries, where $n = \ceil{1/\eps}$.
\end{Corollary}
\begin{proof} 
The statement follows from Theorem~\ref{rent} because a minimal mode query can be simulated by $d-1$ binary mode queries.
(For a fixed index $i_0\in\{1,\ldots,d-1\}$ and a point $x\in A$, in binary mode we successively ask which of $A_{i_01}$, \dots, $A_{i_0(d-1)}$ contains~$x$. If none of them does, then $A_{i_0d}$ contains~$x$.)
\end{proof}

\begin{Remark}
\label{rem:pie-min}
For LPS case~(ii), the above results give estimates for both binary and minimal modes (Theorems~\ref{rent-binary-direct} and~\ref{rent} and Corollary~\ref{rent-binary}), while for LPS case~(i) only a binary mode estimate is given (Theorem~\ref{pie}).
For LPS case~(i) in minimal mode, one can apply a modification of the method described for this mode in the proof of Theorem~\ref{rent}.
We believe that the corresponding algorithm (for LPS case~(i) in minimal mode) will be more complicated, because instead of subsimplices it will involve more complex subsets of the simplex~$A$ (for example, at the first step of the algorithm we move not to a subsimplex, but to the closure of the complement of a subsimplex of~$A$).
As a convenient coordinate system, here we can use the auxiliary simplex $A^+$ of diameter~$2$, the centers of the facets of which are the vertices of~$A$.
Then the subspaces of interest to us are intersections of~$A$ with subsimplices in $A^+$.
We conjecture that with this method we can find an $\eps$-fair division point 
(in minimal mode for LPS case~(i)) 
using at most ${(d-1)}\ceil{\log_{\frac{d}{d-1}} k}$ queries, where $k = \ceil{2/\eps}$.
For the case of small $d$, the validity of this estimate is easy to verify.
We leave a detailed analysis of this algorithm for further research, along with other possible generalizations and variations.
\end{Remark}


\bigskip

\section{Case of convex preference sets}

In this section we deal with the case where the preference sets $A_{ij}$ in the rental harmony problem are convex.
We use the standard notion of convexity in Euclidean spaces: a subset is convex if and only if given any two points in the subset, the subset contains the whole line segment that joins them, or, equivalently, if the subset is the intersection of Euclidean half-spaces.

If $d = 2$ then convex sets containing a fixed vertex of the simplex are intersections of the simplex with half-spaces containing this vertex, which provides a trivial solution with $\ceil{\log_2 n}$ queries (this may be viewed as a trivial case of Theorem~$\ref{rent}$).
Here, we focus on the situation with $d = 3$.

First, we prove an analogue of Theorem~\ref{d-1_enough} for the case where our sets are convex (while do not necessarily have the inclusion property) and we know all $d$ points in intersections (instead of $d-1$).

\begin{Lm}
\label{separate_players}
    Suppose all preference sets $A_{ij}$ are convex and $F_j \subset A_{ij}$ for all $i$ and~$j$. 
    Then there exists an algorithm that, given $\eps>0$ and a collection $x_1, \ldots, x_d$ with $x_i \in I^{\eps}_i$ for each~$i$, finds an $\eps$-fair division point $x = x(x_1, \ldots, x_d)$.
\end{Lm}
\begin{proof}
    First of all, observe that $A^{\eps}_{ij}$ are also convex (because 
the $\eps$\nobreakdash-neighborhood of a set is the Minkowski sum of this set and the $\eps$\nobreakdash-neighborhood of the origin,    
    the $\eps$\nobreakdash-neighborhood of a point is convex in our metric, 
    and the Minkowski sum of convex sets is convex). 
    Put $A'_{ij} = \conv{F_j, x_i}$. Observe that each $A'_{ij}$ depends only on $x_i$ and $j$, but not on $A_{ij}$. 
    Let $\tau$ be any cyclic permutation of $(1,\ldots,d)$, and set $v'_k=v_{\tau(k)}$ for all $k\in (1,\ldots,d)$.
    It can be easily seen that, for each~$i$, the collection $\{A'_{i1}, \ldots, A'_{id}\}$ is a KKM-covering for $A$ with respect to the vertex labelling $\{v'_1,\dots,v'_d\}$ (see Sec.~\ref{Sec: Stating the problem}).
    Then the rainbow KKM lemma implies that there exists a point $z$ and a permutation~$\sigma$ of $(1,\ldots,d)$ such that $z \in A'_{i\sigma(i)}$ for all~$i$. 
    Moreover, it is clear that there exists an algorithm finding~$z$: for example, by checking all permutations of $(1,\ldots,d)$ we can find a permutation $\sigma$ such that
    $$ 
        \bigcap_{1 \le i \le d} A'_{i\sigma(i)} \neq \emptyset
    $$
    and take any point in this intersection.
    But $F_j \subset A^{\eps}_{ij}$ and $x_i \in I^{\eps}_i \subset A^{\eps}_{ij}$ imply that $A'_{ij} \subset A^{\eps}_{ij}$ (since $A'_{ij} = \conv{F_j, x_i}$ and $A^{\eps}_{ij}$ is convex). 
    Thus $z$ is also an $\eps$-fair division point for the initial problem. 
\end{proof}

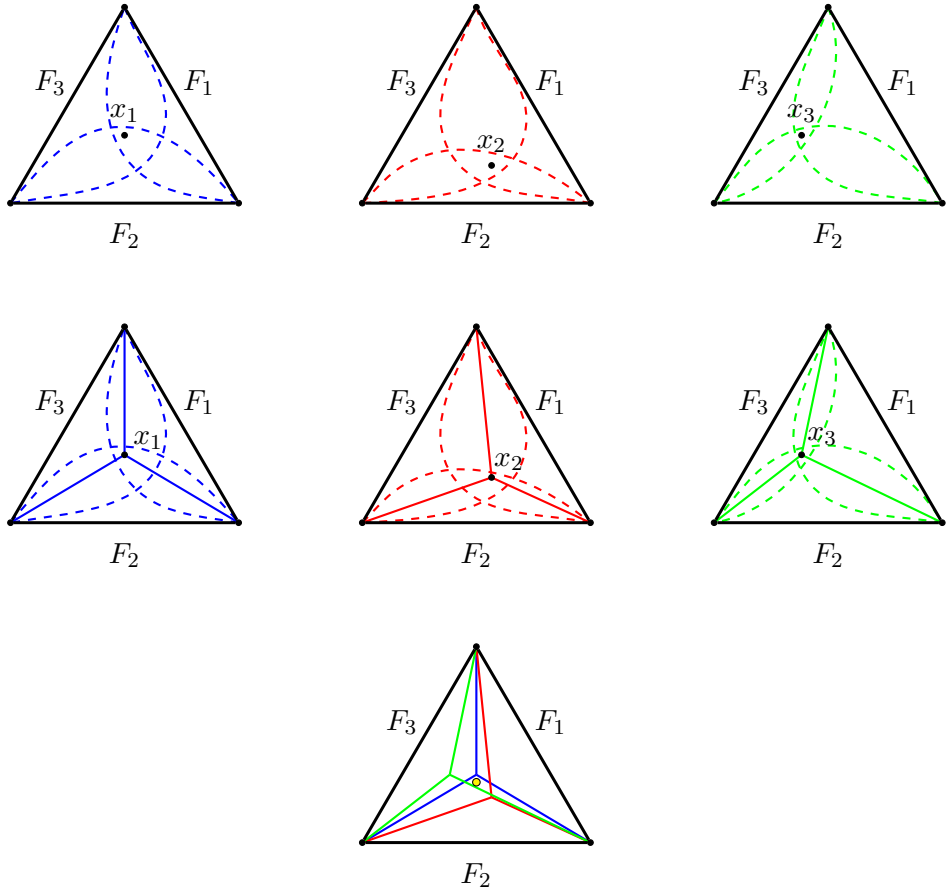
\begin{figure}[h!]
    \centering       
     \begin{minipage}[h]{0.3\linewidth}\centering\begin{tikzpicture}
        [
        scale=1,
        >=stealth,
        point/.style = {draw, circle,  fill = black, inner sep = 1pt},
        dot/.style   = {draw, circle,  fill = black, inner sep = .75pt},
        use Hobby shortcut
        ]
        
        \def\xang{0.500}
        \def\yang{0.866}
        \coordinate (ptA) at (0, 0);
        \coordinate (ptB) at (3, 0);
        \coordinate (ptC) at (3*\xang, 3*\yang);
       
        \draw[blue, thick, dashed] ([out angle=50]ptA)  .. (1.3, 1)   .. ([in angle=130]ptB);
        \draw[blue, thick, dashed] ([out angle=170]ptB) .. (1.7, 0.4) .. ([in angle=250]ptC);
        \draw[blue, thick, dashed] ([out angle=290]ptC) .. (2, 0.9)   .. ([in angle=10]ptA);
        
        \node (A) at (0,0)                                            [dot]{};
        \node (B) at (+1.000 + 1.000 + 1.000, +0.000 + 0.000 + 0.000) [dot]{};
        \node (C) at (+\xang + \xang + \xang, +\yang + \yang + \yang) [dot]{}; 
        \node (x_1) at (1.5, 0.9) [dot, label=$x_1$]{};
        \foreach \from/\to in {A/B, B/C, C/A}
        {
            \draw[line width=0.4mm] (\from) -- (\to);
        }
        \node (F1) at (2, 1.6) [label = right:$F_{1}$]{};
        \node (F2) at (1.5, 0) [label = below:$F_{2}$]{};
        \node (F3) at (1, 1.6) [label = left:$F_{3}$]{};
        \end{tikzpicture}\end{minipage}
        \begin{minipage}[h]{0.3\linewidth}\centering\begin{tikzpicture}
        [
        scale=1,
        >=stealth,
        point/.style = {draw, circle,  fill = black, inner sep = 1pt},
        dot/.style   = {draw, circle,  fill = black, inner sep = .75pt},
        use Hobby shortcut
        ]
        
        \def\xang{0.500}
        \def\yang{0.866}
        
        \coordinate (ptA) at (0, 0);
        \coordinate (ptB) at (3, 0);
        \coordinate (ptC) at (3*\xang, 3*\yang);
        
        \draw[red, thick, dashed] ([out angle=45]ptA)  .. (0.4, 0.4) .. ([in angle=140]ptB);
        \draw[red, thick, dashed] ([out angle=168]ptB) .. (1.2, 0.6) .. ([in angle=246]ptC);
        \draw[red, thick, dashed] ([out angle=285]ptC) .. (2.15, 1)  .. ([in angle=3]ptA);
        \node (A) at (ptA) [dot]{};
        \node (B) at (ptB) [dot]{};
        \node (C) at (ptC) [dot]{}; 
        \node (x_2) at (1.7, 0.5) [dot, label=$x_2$]{};
        \foreach \from/\to in {A/B, B/C, C/A}
        {
            \draw[line width=0.4mm] (\from) -- (\to);
        }
        \node (F1) at (2, 1.6) [label = right:$F_{1}$]{};
        \node (F2) at (1.5, 0) [label = below:$F_{2}$]{};
        \node (F3) at (1, 1.6) [label = left:$F_{3}$]{};
        \end{tikzpicture}\end{minipage}
        \begin{minipage}[h]{0.3\linewidth}\centering\begin{tikzpicture}
        [
        scale=1,
        >=stealth,
        point/.style = {draw, circle,  fill = black, inner sep = 1pt},
        dot/.style   = {draw, circle,  fill = black, inner sep = .75pt},
        use Hobby shortcut
        ]
        
        \def\xang{0.500}
        \def\yang{0.866}
        \coordinate (ptA) at (0, 0);
        \coordinate (ptB) at (3, 0);
        \coordinate (ptC) at (3*\xang, 3*\yang);
        
        \draw[green, thick, dashed] ([out angle=50] ptA) .. (0.9, 0.9) .. ([in angle=120]ptB);
        \draw[green, thick, dashed] ([out angle=170]ptB) .. (1.2, 0.6) .. ([in angle=250]ptC);
        \draw[green, thick, dashed] ([out angle=290]ptC) .. (1.2, 0.8) .. ([in angle=10]ptA);
        
        \node (A) at (ptA) [dot]{};
        \node (B) at (ptB) [dot]{};
        \node (C) at (ptC) [dot]{};
        
        \node (x_3) at (1.15, 0.9) [dot, label=$x_3$]{};
        
        \foreach \from/\to in {A/B, B/C, C/A}
        {
            \draw[line width=0.4mm] (\from) -- (\to);
        } 
        \node (F1) at (2, 1.6) [label = right:$F_{1}$]{};
        \node (F2) at (1.5, 0) [label = below:$F_{2}$]{};
        \node (F3) at (1, 1.6) [label = left:$F_{3}$]{};
        \end{tikzpicture}\end{minipage}
        
        \vspace{24pt}
        
        \begin{minipage}[h]{0.3\linewidth}\centering\begin{tikzpicture}
        [
        scale=1,
        >=stealth,
        point/.style = {draw, circle,  fill = black, inner sep = 1pt},
        dot/.style   = {draw, circle,  fill = black, inner sep = .75pt},
        use Hobby shortcut
        ]
        
        \def\xang{0.500}
        \def\yang{0.866}
        \coordinate (ptA) at (0, 0);
        \coordinate (ptB) at (3, 0);
        \coordinate (ptC) at (3*\xang, 3*\yang);
        \coordinate (x1c) at (1.5, 0.9);

        \draw[blue, thick, dashed] ([out angle=50]ptA)  .. (1.3, 1)   .. ([in angle=130]ptB);
        \draw[blue, thick, dashed] ([out angle=170]ptB) .. (1.7, 0.4) .. ([in angle=250]ptC);
        \draw[blue, thick, dashed] ([out angle=290]ptC) .. (2, 0.9)   .. ([in angle=10]ptA);
        \draw[blue, thick] (ptA) -- (x1c);
        \draw[blue, thick] (ptB) -- (x1c);
        \draw[blue, thick] (ptC) -- (x1c);
        
        \node (A) at (0,0)                                            [dot]{};
        \node (B) at (+1.000 + 1.000 + 1.000, +0.000 + 0.000 + 0.000) [dot]{};
        \node (C) at (+\xang + \xang + \xang, +\yang + \yang + \yang) [dot]{}; 
        \node (x_1) at (x1c) [dot, label={[label distance=-0.1cm]70:$x_1$}]{};
        
        \foreach \from/\to in {A/B, B/C, C/A}
        {
            \draw[line width=0.4mm] (\from) -- (\to);
        }
        \node (F1) at (2, 1.6) [label = right:$F_{1}$]{};
        \node (F2) at (1.5, 0) [label = below:$F_{2}$]{};
        \node (F3) at (1, 1.6) [label = left:$F_{3}$]{};
        \end{tikzpicture}\end{minipage}
         \begin{minipage}[h]{0.3\linewidth}\centering\begin{tikzpicture}
        [
        scale=1,
        >=stealth,
        point/.style = {draw, circle,  fill = black, inner sep = 1pt},
        dot/.style   = {draw, circle,  fill = black, inner sep = .75pt},
        use Hobby shortcut
        ]
        
        \def\xang{0.500}
        \def\yang{0.866}
        
        \coordinate (ptA) at (0, 0);
        \coordinate (ptB) at (3, 0);
        \coordinate (ptC) at (3*\xang, 3*\yang);
        \coordinate (x2c) at (1.7, 0.6);
        
        \draw[red, thick, dashed] ([out angle=45]ptA)  .. (0.4, 0.4) .. ([in angle=140]ptB);
        \draw[red, thick, dashed] ([out angle=168]ptB) .. (1.2, 0.6) .. ([in angle=246]ptC);
        \draw[red, thick, dashed] ([out angle=285]ptC) .. (2.15, 1)  .. ([in angle=3]ptA);
        \draw[red, thick] (ptA) -- (x2c);
        \draw[red, thick] (ptB) -- (x2c);
        \draw[red, thick] (ptC) -- (x2c);
        \node (A) at (ptA) [dot]{};
        \node (B) at (ptB) [dot]{};
        \node (C) at (ptC) [dot]{}; 
        \node (x_2) at (x2c) [dot, label={[label distance=-0.15cm]380:$x_2$}]{};
        
        \foreach \from/\to in {A/B, B/C, C/A}
        {
            \draw[line width=0.4mm] (\from) -- (\to);
        }
        \node (F1) at (2, 1.6) [label = right:$F_{1}$]{};
        \node (F2) at (1.5, 0) [label = below:$F_{2}$]{};
        \node (F3) at (1, 1.6) [label = left:$F_{3}$]{};
        \end{tikzpicture}\end{minipage}
        \begin{minipage}[h]{0.3\linewidth}\centering\begin{tikzpicture}
        [
        scale=1,
        >=stealth,
        point/.style = {draw, circle,  fill = black, inner sep = 1pt},
        dot/.style   = {draw, circle,  fill = black, inner sep = .75pt},
        use Hobby shortcut
        ]
        
        \def\xang{0.500}
        \def\yang{0.866}
        \coordinate (ptA) at (0, 0);
        \coordinate (ptB) at (3, 0);
        \coordinate (ptC) at (3*\xang, 3*\yang);
        \coordinate (x3c) at (1.15, 0.9);

        \draw[green, thick, dashed] ([out angle=50] ptA) .. (0.9, 0.9) .. ([in angle=120]ptB);
        \draw[green, thick, dashed] ([out angle=170]ptB) .. (1.2, 0.6) .. ([in angle=250]ptC);
        \draw[green, thick, dashed] ([out angle=290]ptC) .. (1.2, 0.8) .. ([in angle=10]ptA);
        \draw[green, thick] (ptA) -- (x3c);
        \draw[green, thick] (ptB) -- (x3c);
        \draw[green, thick] (ptC) -- (x3c);
        
        \node (A) at (ptA) [dot]{};
        \node (B) at (ptB) [dot]{};
        \node (C) at (ptC) [dot]{};
        
        \node (x_3)  at (x3c) [dot, label={[label distance=-0.1cm]3:$x_3$}]{};
        
        \foreach \from/\to in {A/B, B/C, C/A}
        {
            \draw[line width=0.4mm] (\from) -- (\to);
        } 
        \node (F1) at (2, 1.6) [label = right:$F_{1}$]{};
        \node (F2) at (1.5, 0) [label = below:$F_{2}$]{};
        \node (F3) at (1, 1.6) [label = left:$F_{3}$]{};
        \end{tikzpicture}\end{minipage}
        
        \vspace{24pt}
            \begin{minipage}[h]{0.3\linewidth}\centering\begin{tikzpicture}
        [
        scale=1,
        >=stealth,
        point/.style = {draw, circle,  fill = black, inner sep = 1pt},
        dot/.style   = {draw, circle,  fill = black, inner sep = .75pt},
        use Hobby shortcut
        ]
        
        \def\xang{0.500}
        \def\yang{0.866}
        \coordinate (ptA) at (0, 0);
        \coordinate (ptB) at (3, 0);
        \coordinate (ptC) at (3*\xang, 3*\yang);
        \coordinate (x1c) at (1.5, 0.9);
        \coordinate (x2c) at (1.7, 0.6);
        \coordinate (x3c) at (1.15, 0.9);

        \draw[blue, thick]  (ptA) -- (x1c);
        \draw[blue, thick]  (ptB) -- (x1c);
        \draw[blue, thick]  (ptC) -- (x1c);
        \draw[red, thick]   (ptA) -- (x2c);
        \draw[red, thick]   (ptB) -- (x2c);
        \draw[red, thick]   (ptC) -- (x2c);
        \draw[green, thick] (ptA) -- (x3c);
        \draw[green, thick] (ptB) -- (x3c);
        \draw[green, thick] (ptC) -- (x3c);
        
        \node (A) at (ptA) [dot]{};
        \node (B) at (ptB) [dot]{};
        \node (C) at (ptC) [dot]{};
        \foreach \from/\to in {A/B, B/C, C/A}
        {
            \draw[line width=0.4mm] (\from) -- (\to);
        }
        
        \node (F1) at (2, 1.6) [label = right:$F_{1}$]{};
        \node (F2) at (1.5, 0) [label = below:$F_{2}$]{};
        \node (F3) at (1, 1.6) [label = left:$F_{3}$]{};
        \node (X) at (1.5, 0.8) [draw, circle,  fill = yellow, inner sep = 1pt]{};
        \end{tikzpicture}\end{minipage}

    \caption{An example of new covering for Lemma~\ref{separate_players}.}
    \label{fig:exampleforlemma}
\end{figure}

\begin{Th}
\label{convex}
    Suppose that $d = 3$ and each of the sets $A_{ij}$ is a convex set containing~$F_j$. 
    Then we can find an $\eps$-fair division point in binary mode with at most $6(\ceil{\log_2 n}^2 + \ceil{\log_2 n}) = O(\log^2 n)$ queries, where $n = \ceil{1/\eps}$. 
    
\end{Th}
\begin{proof}
    By Lemma~\ref{separate_players} it is sufficient to find a point in each of $I^{\eps}_i$, so we want to do this for an arbitrary fixed~$i$ with at most $2(\ceil{\log_2 n}^2 + \ceil{\log_2 n})$ queries.
    
    Based on the barycentric coordinates notation introduced in Sec.~\ref{Sec: Stating the problem}, for the triples of numbers $(a,b,c)$ with nonzero sum $a+b+c$ we use the notation 
    $$
    [a, b, c]_*= [a/(a+b+c), b/(a+b+c), c/(a+b+c)]
    $$
    and introduce the grid
    $$
        \{ [a, b, c]_* : a, b, c \in \mathbb{N}_0, a + b + c = n \}.
    $$
    Our queries will only be about grid points. We define 
    \begin{gather*}
        k_2(a) = \max\{ b \in \mathbb{N}_0 : 0 \le b \le n - a, [a, b, n - a - b]_* \in A_{i2} \},\\
        k_3(a) = \min\{ b \in \mathbb{N}_0 : 0 \le b \le n - a, [a, b, n - a - b]_* \in A_{i3}\}.
    \end{gather*}
    
    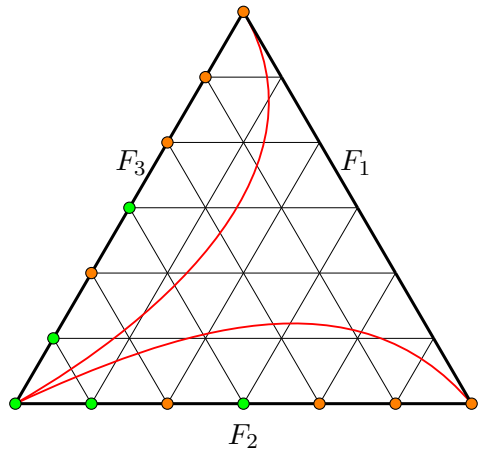
\begin{figure}[H]
        \centering
        
        \begin{tikzpicture}
        [
        scale= 1,
        >=stealth,
        point/.style = {draw, circle,  fill = black, inner sep = 1pt},
        covered/.style = {draw, circle,  fill = green, inner sep = 1.5pt},
        uncovered/.style = {draw, circle,  fill = orange, inner sep = 1.5pt},
        dot/.style   = {draw, circle,  fill = black, inner sep = .5pt},
        use Hobby shortcut
        ]
        
        \def\xang{0.500}
        \def\yang{0.866}

        \coordinate (z)  at (0,0);
        
        \coordinate (z0) at (+1.000, +0.000);
        \coordinate (z1) at (+\xang, +\yang);
        
        \coordinate (z00) at (+1.000 + 1.000, +0.000 + 0.000);
        \coordinate (z10) at (+\xang + 1.000, +\yang + 0.000);
        \coordinate (z11) at (+\xang + \xang, +\yang + \yang);
        
        \coordinate (z000) at (+1.000 + 1.000 + 1.000, +0.000 + 0.000 + 0.000);
        \coordinate (z100) at (+\xang + 1.000 + 1.000, +\yang + 0.000 + 0.000); 
        \coordinate (z110) at (+\xang + \xang + 1.000, +\yang + \yang + 0.000); 
        \coordinate (z111) at (+\xang + \xang + \xang, +\yang + \yang + \yang);
        
        \coordinate (z0000) at (+1.000 + 1.000 + 1.000 + 1.000, +0.000 + 0.000 + 0.000 + 0.000);
        \coordinate (z1000) at (+\xang + 1.000 + 1.000 + 1.000, +0.000 + \yang + 0.000 + 0.000); 
        \coordinate (z1100) at (+\xang + \xang + 1.000 + 1.000, +0.000 + \yang + \yang + 0.000); 
        \coordinate (z1110) at (+\xang + \xang + \xang + 1.000, +\yang + \yang + \yang + 0.000);
        \coordinate (z1111) at (+\xang + \xang + \xang + \xang, +\yang + \yang + \yang + \yang);

        \coordinate (z00000) at (+1.000 + 1.000 + 1.000 + 1.000 + 1.000, +0.000 + 0.000 + 0.000 + 0.000 + 0.000);
        \coordinate (z10000) at (+\xang + 1.000 + 1.000 + 1.000 + 1.000, +0.000 + \yang + 0.000 + 0.000 + 0.000);
        \coordinate (z11000) at (+\xang + \xang + 1.000 + 1.000 + 1.000, +0.000 + \yang + \yang + 0.000 + 0.000); 
        \coordinate (z11100) at (+\xang + \xang + \xang + 1.000 + 1.000, +\yang + \yang + \yang + 0.000 + 0.000);
        \coordinate (z11110) at (+\xang + \xang + \xang + \xang + 1.000, +\yang + \yang + \yang + \yang + 0.000);
        \coordinate (z11111) at (+\xang + \xang + \xang + \xang + \xang, +\yang + \yang + \yang + \yang + \yang);
        
        \coordinate (z000000) at (+1.000 + 1.000 + 1.000 + 1.000 + 1.000 + 1.000, +0.000 + 0.000 + 0.000 + 0.000 + 0.000 + 0.000);
        \coordinate (z100000) at (+\xang + 1.000 + 1.000 + 1.000 + 1.000 + 1.000, +0.000 + \yang + 0.000 + 0.000 + 0.000 + 0.000);
        \coordinate (z110000) at (+\xang + \xang + 1.000 + 1.000 + 1.000 + 1.000, +0.000 + \yang + \yang + 0.000 + 0.000 + 0.000);
        \coordinate (z111000) at (+\xang + \xang + \xang + 1.000 + 1.000 + 1.000, +\yang + \yang + \yang + 0.000 + 0.000 + 0.000);
        \coordinate (z111100) at (+\xang + \xang + \xang + \xang + 1.000 + 1.000, +\yang + \yang + \yang + \yang + 0.000 + 0.000);
        \coordinate (z111110) at (+\xang + \xang + \xang + \xang + \xang + 1.000, +\yang + \yang + \yang + \yang + \yang + 0.000);
        \coordinate (z111111) at (+\xang + \xang + \xang + \xang + \xang + \xang, +\yang + \yang + \yang + \yang + \yang + \yang);
        
        \foreach \from/\to in {
        z1/z10, z10/z100, z100/z1000, z1000/z10000, z10000/z100000,
        z11/z110, z110/z1100, z1100/z11000, z11000/z110000,
        z111/z1110, z1110/z11100, z11100/z111000,
        z1111/z11110, z11110/z111100,
        z11111/z111110,
        z0/z10, z10/z110, z110/z1110, z1110/z11110, z11110/z111110,
        z00/z100, z100/z1100, z1100/z11100, z11100/z111100,
        z000/z1000, z1000/z11000, z11000/z111000,
        z0000/z10000, z10000/z110000,
        z00000/z100000,
        z00000/z10000, z10000/z11000, z11000/z11100, z11100/z11110, z11110/z11111,
        z0000/z1000, z1000/z1100, z1100/z1110, z1110/z1111,
        z000/z100, z100/z110, z110/z111,
        z00/z10, z10/z11,
        z0/z1}
        {
                \draw[line width=0.1mm] (\from) -- (\to);
        }

        \draw[color=red, line width=0.25mm] (0, 0) to [out=25,in=130] (z000000);
        \draw[color=red, line width=0.25mm] (0, 0) to [out=30,in=300] (z111111);
    
        \foreach \from/\to in {
        z/z0, z0/z00, z00/z000, z000/z0000, z0000/z00000, z00000/z000000,
        z/z1, z1/z11, z11/z111, z111/z1111, z1111/z11111, z11111/z111111,
        z000000/z100000, z100000/z110000, z110000/z111000, z111000/z111100, z111100/z111110, z111110/z111111}
        {
            \draw[line width=0.4mm] (\from) -- (\to);
        }
        
        \node (zz) at (z)  [covered]{};
        \node (x1) at (z0) [covered]{};
        \node (y1) at (z1) [covered]{};
        \node (x2) at (z00) [uncovered]{};
        \node (y2) at (z11) [uncovered]{};
        \node (x3) at (z000) [covered]{};
        \node (y3) at (z111) [covered]{};
        \node (x4) at (z0000) [uncovered]{};
        \node (y4) at (z1111) [uncovered]{};
        \node (x5) at (z00000) [uncovered]{};
        \node (y5) at (z11111) [uncovered]{};
        \node (x6) at (z000000) [uncovered]{};
        \node (y6) at (z111111) [uncovered]{};
        
        \node (F1) at (4, 3.2) [label = right:$F_{1}$]{};
        \node (F2) at (3, 0) [label = below:$F_{2}$]{};
        \node (F3) at (2, 3.2) [label = left:$F_{3}$]{};
        \end{tikzpicture}
        
        \caption{Green points correspond to covered values and orange points correspond to uncovered ones.}
        \label{fig:cover-uncover}
    \end{figure}
    
    Observe that $[a, 0, n - a]_* \in A_{i2}$ and $[a, n - a, 0]_* \in A_{i3}$. 
    Then convexity of $A_{i2}$ and $A_{i3}$ implies that the values of $k_2(a)$ and $k_3(a)$ for any fixed $a$ can be found with binary search using $2\ceil{\log_2 n}$ queries.
    
We say that a coordinate $a$ is \emph{covered} if $k_2(a) + 1 \ge k_3(a)$ and \emph{uncovered} otherwise. In other words, $a$ is covered if the union $A_{i2} \cup A_{i3}$ contains all points of the grid with the first coordinate $a$. Clearly, $a = n$ is covered. 
Then we spend $2 \ceil{\log_2 n}$ queries to calculate $k_2(0)$ and $k_3(0)$ and check whether $a = 0$ is covered. 
If it is then $[0, k_2(a), n - k_2(a)]_*$ is the desired point. 
Otherwise we can run a binary search through $a$ to obtain a coordinate $a_0$ such that it is uncovered while $a_0 + 1$ is covered (note that this is possible despite the set of covered values has ``gaps'' in general): each iteration of binary search requires at most $2 \ceil{\log_2 n}$ queries, so the total search takes at most $2 \ceil{\log_2 n}^2$ ones.

After finding $a_0$, we turn to the following sequence~$Z$ of $k=k_3(a_0)-k_2(a_0)+2$ points (see Fig.~\ref{fig:zpath}):
    \begin{equation*} 
    \begin{array}{rlll}
    z_1:= & [a_0,     & k_2(a_0),   & n-a_0-k_2(a_0)]_*, \\
    z_2:= & {[a_0+1,} & k_2(a_0),   & n-a_0-k_2(a_0)-1]_*, \\
    z_3:= & {[a_0+1,} & k_2(a_0)+1, & n-a_0-k_2(a_0)-2]_*, \\
      & ... & \\
    z_{k-2}:= & {[a_0+1,} & k_3(a_0)-2, & n-a_0-k_3(a_0)+1]_*, \\
    z_{k-1}:= & {[a_0+1,} & k_3(a_0)-1, & n-a_0-k_3(a_0)]_*, \\
    z_{k}:= & {[a_0,} & k_3(a_0), & n-a_0-k_3(a_0)]_*.
    \end{array}
    \end{equation*}

    \begin{figure}[H]
        \centering
        
        \begin{tikzpicture}
        [
        scale= 1,
        >=stealth,
        point/.style = {draw, circle,  fill = black, inner sep = 1pt},
        covered/.style = {draw, circle,  fill = green, inner sep = 1.5pt},
        uncovered/.style = {draw, circle,  fill = orange, inner sep = 1.5pt},
        zpoint/.style = {draw, circle, blue, fill = blue, inner sep = 1pt},
        zpoint_big/.style = {draw, circle, blue, fill = blue, inner sep = 2pt},
        dot/.style   = {draw, circle,  fill = black, inner sep = .5pt},
        use Hobby shortcut
        ]
        
        \def\xang{0.500}
        \def\yang{0.866}

        \coordinate (z)  at (0,0);
        
        \coordinate (z0) at (+1.000, +0.000);
        \coordinate (z1) at (+\xang, +\yang);
        
        \coordinate (z00) at (+1.000 + 1.000, +0.000 + 0.000);
        \coordinate (z10) at (+\xang + 1.000, +\yang + 0.000);
        \coordinate (z11) at (+\xang + \xang, +\yang + \yang);
        
        \coordinate (z000) at (+1.000 + 1.000 + 1.000, +0.000 + 0.000 + 0.000);
        \coordinate (z100) at (+\xang + 1.000 + 1.000, +\yang + 0.000 + 0.000); 
        \coordinate (z110) at (+\xang + \xang + 1.000, +\yang + \yang + 0.000); 
        \coordinate (z111) at (+\xang + \xang + \xang, +\yang + \yang + \yang);
        
        \coordinate (z0000) at (+1.000 + 1.000 + 1.000 + 1.000, +0.000 + 0.000 + 0.000 + 0.000);
        \coordinate (z1000) at (+\xang + 1.000 + 1.000 + 1.000, +0.000 + \yang + 0.000 + 0.000); 
        \coordinate (z1100) at (+\xang + \xang + 1.000 + 1.000, +0.000 + \yang + \yang + 0.000); 
        \coordinate (z1110) at (+\xang + \xang + \xang + 1.000, +\yang + \yang + \yang + 0.000);
        \coordinate (z1111) at (+\xang + \xang + \xang + \xang, +\yang + \yang + \yang + \yang);

        \coordinate (z00000) at (+1.000 + 1.000 + 1.000 + 1.000 + 1.000, +0.000 + 0.000 + 0.000 + 0.000 + 0.000);
        \coordinate (z10000) at (+\xang + 1.000 + 1.000 + 1.000 + 1.000, +0.000 + \yang + 0.000 + 0.000 + 0.000);
        \coordinate (z11000) at (+\xang + \xang + 1.000 + 1.000 + 1.000, +0.000 + \yang + \yang + 0.000 + 0.000); 
        \coordinate (z11100) at (+\xang + \xang + \xang + 1.000 + 1.000, +\yang + \yang + \yang + 0.000 + 0.000);
        \coordinate (z11110) at (+\xang + \xang + \xang + \xang + 1.000, +\yang + \yang + \yang + \yang + 0.000);
        \coordinate (z11111) at (+\xang + \xang + \xang + \xang + \xang, +\yang + \yang + \yang + \yang + \yang);
        
        \coordinate (z000000) at (+1.000 + 1.000 + 1.000 + 1.000 + 1.000 + 1.000, +0.000 + 0.000 + 0.000 + 0.000 + 0.000 + 0.000);
        \coordinate (z100000) at (+\xang + 1.000 + 1.000 + 1.000 + 1.000 + 1.000, +0.000 + \yang + 0.000 + 0.000 + 0.000 + 0.000);
        \coordinate (z110000) at (+\xang + \xang + 1.000 + 1.000 + 1.000 + 1.000, +0.000 + \yang + \yang + 0.000 + 0.000 + 0.000);
        \coordinate (z111000) at (+\xang + \xang + \xang + 1.000 + 1.000 + 1.000, +\yang + \yang + \yang + 0.000 + 0.000 + 0.000);
        \coordinate (z111100) at (+\xang + \xang + \xang + \xang + 1.000 + 1.000, +\yang + \yang + \yang + \yang + 0.000 + 0.000);
        \coordinate (z111110) at (+\xang + \xang + \xang + \xang + \xang + 1.000, +\yang + \yang + \yang + \yang + \yang + 0.000);
        \coordinate (z111111) at (+\xang + \xang + \xang + \xang + \xang + \xang, +\yang + \yang + \yang + \yang + \yang + \yang);
        
        \foreach \from/\to in {
        z1/z10, z10/z100, z100/z1000, z1000/z10000, z10000/z100000,
        z11/z110, z110/z1100, z1100/z11000, z11000/z110000,
        z111/z1110, z1110/z11100, z11100/z111000,
        z1111/z11110, z11110/z111100,
        z11111/z111110,
        z0/z10, z10/z110, z110/z1110, z1110/z11110, z11110/z111110,
        z00/z100, z100/z1100, z1100/z11100, z11100/z111100,
        z000/z1000, z1000/z11000, z11000/z111000,
        z0000/z10000, z10000/z110000,
        z00000/z100000,
        z00000/z10000, z10000/z11000, z11000/z11100, z11100/z11110, z11110/z11111,
        z0000/z1000, z1000/z1100, z1100/z1110, z1110/z1111,
        z000/z100, z100/z110, z110/z111,
        z00/z10, z10/z11,
        z0/z1}
        {
                \draw[line width=0.1mm] (\from) -- (\to);
        }
        
        \foreach \from/\to in {
        z/z0, z0/z00, z00/z000, z000/z0000, z0000/z00000, z00000/z000000,
        z/z1, z1/z11, z11/z111, z111/z1111, z1111/z11111, z11111/z111111,
        z000000/z100000, z100000/z110000, z110000/z111000, z111000/z111100, z111100/z111110, z111110/z111111}
        {
            \draw[line width=0.4mm] (\from) -- (\to);
        }
        
        \draw[red, thick] ([out angle=50] z) .. (1.3, 1.3) .. ([in angle=120] z000000);
        \draw[red, thick] ([out angle=290] z111111) .. (3, 1.2) .. ([in angle=10]z);
        
        \node (zp1) at (z10000) [zpoint]{};
        \node (zp2) at (z1000)  [zpoint]{};
        \node (zp3) at (z1100)  [zpoint]{};
        \node (zp4) at (z11100) [zpoint]{};
        \draw[thick, blue, ->] (zp1) -- (zp2);
        \draw[thick, blue, ->] (zp2) -- (zp3);
        \draw[thick, blue, ->] (zp3) -- (zp4);
        \node (zz) at (z)  [covered]{};
        \node (x1) at (z0) [covered]{};
        \node (y1) at (z1) [covered]{};
        \node (x2) at (z00) [covered]{};
        \node (y2) at (z11) [covered]{};
        \node (x3) at (z000) [covered]{};
        \node (y3) at (z111) [covered]{};
        \node (x4) at (z0000) [covered]{};
        \node (y4) at (z1111) [covered]{};
        \node (x5) at (z00000) [uncovered]{};
        \node (y5) at (z11111) [uncovered]{};
        \node (x6) at (z000000) [uncovered]{};
        \node (y6) at (z111111) [uncovered]{};
        
        \node (F1) at (4, 3.2) [label = right:$F_{1}$]{};
        \node (F2) at (3, 0) [label = below:$F_{2}$]{};
        \node (F3) at (2, 3.2) [label = left:$F_{3}$]{};
        \end{tikzpicture}
        
        \caption{The sequence $Z$ is highlighted in blue.}
        \label{fig:zpath}
    \end{figure}
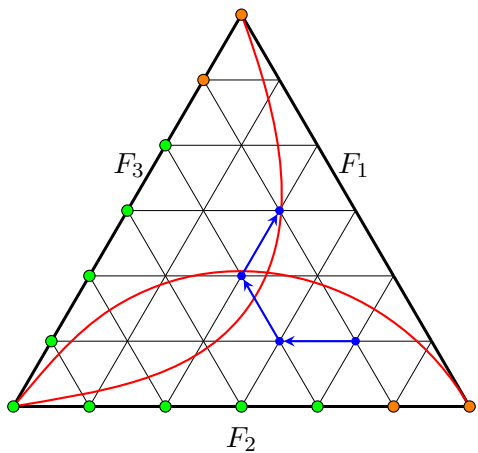
    
Observe that by construction we have $z_1\in A_{i2}$ and $z_k\in A_{i3}$ while the union $A_{i2}\cup A_{i3}$ contains~$Z$. 
This implies that for some $m$ we have $z_m\in A_{i2}$ and $z_{m+1}\in A_{i3}$. 
We can easily find such an $m$ with $z_m\in A_{i2}$ and $z_{m+1}\in A_{i3}$ since we know~$k_2(a_0+1)$ and~$k_3(a_0+1)$.

We claim that $z_m$ and $z_{m+1}$ are in~$I^{\eps}_i$.
Indeed, let $X$ denote the set of points of our grid lying at the level $a_0$ between $z_1$ and $z_k$ (thus, $X$ consists of $k-3$ points).
Observe that $X$ is contained in~$A_{i1}$ because definitions of $z_1$ and $k_2(a_0)$ imply that $X$ does not intersect $A_{i2}$, and definitions of $z_k$ and $k_3(a_0)$ imply that $X$ does not intersect~$A_{i3}$.
Observe also that each pair of consecutive points in~$Z$ is a pair of vertices of a triangle of our grid with its third vertex lying in~$X$. 
Thus, we have $z_m\in A_{i2}$ and $z_{m+1}\in A_{i3}$, while $\{z_m,z_{m+1}\}$ is a pair of vertices of a triangle of our grid with its third vertex (say, $x$) lying in~$X\subset A_{i1}$.
This implies that $z_m$, $z_{m+1}$, and~$x$ are in~$I^{\eps}_i$.
\end{proof}

\begin{Remark}
It would be natural, in conjunction with the ``rental harmony case'' of Theorem~\ref{convex}, to study the ``cake-cutting case'' where each preference set is a convex set containing a vertex of the triangle.
An obstacle to the direct transfer of methods from the proof of Theorem~\ref{convex} to the case with convex preference sets containing vertices is, in particular, a difference in the geometric combinatorics of sets of these types.
For example, if a convex set $A_1$ contains an edge of a triangle, and a convex set $A_2$ contains another edge of this triangle, then the set $A_2\setminus A_1$ is connected.
On the contrary, if a convex set~$B_1$ contains a vertex of a triangle, and a convex set~$B_2$ contains another vertex of this triangle, then the set $B_2\setminus B_1$ can have arbitrarily large number of components.
\end{Remark}


\medskip 

\noindent {\bf Acknowledgements}. We wish to thank Florian Frick, Fr\'ed\'eric Meunier, and Herv\'e Moulin for helpful discussions and comments. 
We also thank the referee for many useful suggestions that significantly improved the previous version of this paper and, in particular, allowed us to shorten some of the proofs.





\end{document}